\documentclass{article}

\usepackage[english]{babel}
\usepackage[utf8x]{inputenc}
\usepackage[T1]{fontenc}
\usepackage[a4paper,top=3cm,bottom=2cm,left=3cm,right=3cm,marginparwidth=1.75cm]{geometry}
\usepackage{amsmath}
\usepackage{graphicx}
\usepackage[colorinlistoftodos]{todonotes}
\usepackage{listings}
\usepackage{color} 
\definecolor{mygreen}{RGB}{28,172,0} 
\definecolor{mylilas}{RGB}{170,55,241}
\usepackage{changepage}
\usepackage{graphicx}
\usepackage{caption}
\usepackage{subcaption}
\usepackage[makeroom]{cancel}
\usepackage{listings}
\usepackage{amssymb}
\usepackage{amsmath}
\usepackage{amsfonts}
\usepackage{float}
\usepackage{listings}
\usepackage{framed}
\usepackage[outdir=./]{epstopdf}
\usepackage{amsthm}
\usepackage{algorithm}
\usepackage{pdfpages}
\usepackage{mathtools}
\usepackage{bm}
\usepackage{mathrsfs}
\usepackage{enumerate}
\usepackage{authblk}
\usepackage{todonotes}

\newtheorem{theorem}{Theorem}
\newtheorem{proposition}{Proposition}
\newtheorem{corollary}{Corollary}

\newtheorem{assumption}{Assumption}
\newtheorem{lemma}{Lemma}

\theoremstyle{definition}
\newtheorem{definition}{Definition}
\newtheorem{remark}{Remark}
\newtheorem{example}{Example}

\usepackage{mathmacros}
\newcommand{\prob}{\mathscr{P}} 
\newcommand{\cgf}{\mathcal{K}} 
\newcommand{\nd}{\mathcal{N}} 
\newcommand{\match}{\mathcal{M}} 
\newcommand{\proj}{\Pi} 
\newcommand{\overbar}[1]{\mkern 1.5mu\overline{\mkern-1.5mu#1\mkern-1.5mu}\mkern 1.5mu}

\providecommand{\keywords}[1]{\textit{Keywords and phrases: } #1}

\begin{document}

\title{Convergence and stability of a micro-macro acceleration method: linear slow-fast stochastic differential equations with additive noise}

\author[1]{Przemys{\l}aw Zieli{\'n}ski}
\author[1]{Hannes Vandecasteele}
\author[1]{Giovanni Samaey}
\affil[1]{KU Leuven, Department of Computer Science, NUMA Section, Celestijnenlaan 200A box 2402, 3001 Leuven, Belgium}
\date{\today}

\maketitle
\begin{abstract}
We analyse the convergence and stability of a micro-macro acceleration algorithm for Monte Carlo simulations of stiff stochastic differential equations with a time-scale separation between the fast evolution of the individual stochastic realizations and some slow macroscopic state variables of the process. The micro-macro acceleration method performs a short simulation of a large ensemble of individual fast paths, before extrapolating the macroscopic state variables of interest over a larger time step. After extrapolation, the method constructs a new probability distribution that is consistent with the extrapolated macroscopic state variables, while minimizing Kullback-Leibler divergence with respect to the distribution available at the end of the Monte Carlo simulation. In the current work, we study the convergence and stability of this method on linear stochastic differential equations with additive noise, when only extrapolating the mean of the slow component. For this case, we prove convergence to the microscopic dynamics when the initial distribution is Gaussian and present a stability result for non-Gaussian initial laws.
\end{abstract}

\section{Introduction}
Applications with multiple time scales arise in many domains, such as nanoscience \cite{fish2006bridging}, fluid dynamics \cite{steinhauser2017computational}, material science \cite{kwon2008multiscale} and life sciences \cite{deisboeck2010multiscale}. Still, the design and analysis of efficient numerical schemes for stiff stochastic differential equations (SDEs) remains challenging. On the one hand, due to a small stability domain, explicit schemes require too many time steps to reach the end of the simulation. On the other hand, while implicit methods are successful for ordinary differential equations, they fail to compute the correct invariant distribution for SDEs \cite{li2008effectiveness}.  Therefore, stiff SDEs require new dedicated numerical multiscale methods. A lot of work has already been done over the years, and we refer to the heterogeneous multiscale method \cite{weinan2003multiscale,abdulle2012heterogeneous}, equation-free techniques \cite{kevrekidis2003equation,kevrekidis2009equation}, and S-ROCK \cite{abdulle2007stabilized,abdulle2008s} in particular, as starting points in the literature.

Recently, a micro-macro acceleration algorithm was introduced to accelerate the Monte Carlo simulation of SDEs with a time-scale separation between the fast individual stochastic paths and some slow macroscopic state variables of interest, which we define as expectations of some quantities of interest over the microscopic distributions~\cite{debrabant2017micro}. Micro-macro acceleration connects the two levels of description of the process: individual paths of the SDE model on the microscopic level and the macroscopic state variables at the macroscopic level. The method alleviates the computational cost of a direct Monte Carlo simulation by interleaving a short bursts of Monte Carlo simulation with extrapolation of the macroscopic state variables over a larger time step. After extrapolation, the method constructs a new probability distribution by \textit{matching} the last available distribution after the microscopic simulation with the extrapolated macroscopic state variables. Matching minimally perturbs the last available distribution after the microscopic simulation to make it consistent with the extrapolated state variables. Thus, in this approach, matching is inherently an optimization problem. There are many ways of measuring the difference between probability distributions. Following~\cite{lelievre2018analysis}, we choose relative entropy or Kullback-Leibler divergence in the present paper, based on information theoretic considerations; for some other strategies see~\cite{debrabant2017micro}.  

A few fundamental properties of micro-macro acceleration have already been investigated. First, in \cite{lelievre2018analysis}, convergence was studied for general SDEs and for any time-scale separation. The analysis requires fixing an infinite hierarchy of macroscopic state variables that are used for extrapolation. It is then shown that convergence not only depends on taking the micro and extrapolation time steps to zero, but also on extrapolating an increasing number of macroscopic state variables as the extrapolation time step decreases. Second, in \cite{debrabant2018study}, the asymptotic numerical stability of the micro-macro acceleration method was studied on a linear system of SDEs with Gaussian initial conditions.
The stability criterion that was employed checks if the distributions obtained from the micro-macro acceleration method reach the equilibrium Gaussian distribution of the underlying microscopic integrator as the number of micro-macro time-steps tends to infinity. The analysis reveals that, when the slow and fast components of the system are decoupled, the maximal extrapolation time step is independent of the time-scale separation. 

Convergence and stability are concerned with two limiting situations: convergence studies the method's behaviour at a fixed moment in time as the extrapolation time step tends to zero, whereas stability studies the method's behaviour for large time steps and long time horizons. Once both convergence and stability have been established, one can look at the appropriate selection of the extrapolation time step and the number of macroscopic state variables for accuracy and efficiency. In \cite{vandecasteele2018efficiency}, we recently investigated the accuracy and efficiency of the micro-macro acceleration on slow-fast systems, 
 showing numerically that micro-macro acceleration can simultaneously take larger time steps than the microscopic time integrator, while obtaining a smaller error than approximate macroscopic models for the slow component of the system. 

In this work, we expand the convergence and stability study of the micro-macro acceleration method on linear slow-fast SDEs with additive noise. For these equations, we prove convergence to the microscopic dynamics for Gaussian initial conditions when only extrapolating the mean of the slow component. To this end, we look at the propagation of the mean and variance of the full microscopic state throughout the micro-macro acceleration scheme, only extrapolating the mean of the slow component of the SDE. We prove that, when the extrapolation time step goes to zero, these two first moments converge to the corresponding ones produced by the underlying Euler-Maruyama scheme. Dragging the microscopic time step to zero, we further obtain convergence to the exact dynamics of the slow-fast SDE. In contrast to the convergence result from~\cite{lelievre2018analysis}, this analysis does not require any hierarchy of macroscopic state variables; both the mean of the fast component and the variance of the SDE are never extrapolated.

We also present a stability result for non-Gaussian initial laws that complements the previous findings in the Gaussian framework~\cite{debrabant2018study}. For this part, we consider a class of (non-Gaussian) initial conditions having Gaussian tails. Inspired by the case of Gaussian initial conditions analyzed in~\cite{debrabant2018study}, where the explicit formulas are available, we prove that the stability of the micro-macro acceleration method hinges on the stability of the mean obtained from the matching procedure. In this case, however, due to the non-Gaussianity of distributions, we have to look at the propagation of all higher moments throughout the method.

These findings illustrate a certain robustness of the matching procedure: solely using the first moment, matching reconstructs the remaining higher moments, so that we converge to the microscopic scheme as extrapolation vanishes, and recover its invariant distribution as the number of fixed extrapolations grows to infinity. 

\vspace{2mm}
\paragraph{\textbf{Slow-fast linear SDEs with additive noise}}
The objects of interest in this manuscript are linear SDEs with additive noise, or Ohrnstein-Uhlenbeck processes, of the form
\begin{equation} \label{eq:linearsde}
dX_t = A X_t dt + B dW_t
\end{equation}
with a square drift matrix $A \in \mathbb{R}^{d \times d}$, a rectangular diffusion matrix $B \in \mathbb{R}^{d \times m}$ and the Wiener process $W_t \in \mathbb{R}^m$. There are a few reasons why linear SDEs with additive noise are useful to study. First, the dynamics of linear systems is well understood, and we can derive stronger convergence results of micro-macro acceleration on such systems. Second, Ohrnstein-Uhlenbeck processes have an invariant distribution. We can then investigate whether micro-macro acceleration converges in distribution to the correct equilibrium distribution, as was done in \cite{debrabant2018study} as a function of the time-scale separation and the extrapolation step size. Third, linear systems are popular in the context of ordinary differential equations to determine the stability of deterministic methods by an eigenvalue analysis of the drift matrix $A$. Although the concept of linearization is ambiguously defined in the stochastic case, linear SDEs with additive noise are useful to study in their own~right.

In this manuscript, we are concerned with linear SDEs with a time-scale separation between some slow and some fast variables. A spectral gap in the drift matrix $A$ is a good indication of a time-scale separation present in the linear system. The larger the gap, the more stiff the linear system \eqref{eq:linearsde} becomes. In the context of micro-macro acceleration, we are mainly interested in the evolution of some moments of the slow components of \eqref{eq:linearsde}. Using the spectral decomposition theorem, we introduce the orthogonal projections $\Pi^s \from \mathbb{R}^d \to \mathbb{R}^{d_s}$ and $\Pi^f \from \mathbb{R}^d \to \mathbb{R}^{d-d_s}$ that map the full state space onto the `slow'~$\R^{d_s}$ and `fast'~$\R^{d-d_s}$ state spaces, respectively, which correspond to the gap in the spectrum of $A$.  Such a procedure is also called `coarse-graining' \cite{debrabant2018study}. The decomposition is such that we can express the full state space and the drift matrix as
\[
\mathbb{R}^d = \mathbb{R}^{d_s} \oplus \mathbb{R}^{d-d_s},\quad A =  \Pi^s A^s \oplus \Pi^f A^f.
\]
 We aim at approximating the moments of the projected process $\Pi^s X_t$ as well as possible to compute the exact evolution of these moments with a reasonable accuracy. For the remainder of the manuscript, a superscript `s' denote the slow components and a superscript `f' the fast.


\vspace{2mm}
\paragraph{\textbf{Outline of the paper}}

The paper is organized as follows: in Section~\ref{sec:mM_acc}, we introduce the micro-macro acceleration algorithm, specifically in our context of linear SDEs with additive noise. 
Section~\ref{sec:convergence} contains the proof of convergence of micro-macro acceleration to the complete microscopic dynamics with only slow mean extrapolation. Section~\ref{sec:stability} investigates the stability of the method when the initial condition has Gaussian tails, when only extrapolating slow mean. In Section~\ref{sec:numerics}, we illustrate the theoretical results on convergence and stability with numerical examples. We consider a system of linear SDEs with additive noise, with an extra periodic force on the slow component.  Section~\ref{sec:conclusions} contains a concluding discussion.

\section{The micro-macro acceleration algorithm\label{sec:mM_acc}}

One cycle of the micro-macro acceleration consists of four parts: (i) a microscopic \textit{simulation} of the (stiff) stochastic differential equation over a short time interval, discretized with small time steps $\delta t$; (ii) \textit{restriction} or computing an estimate of the macroscopic state variables based on the microscopic ensembles from (i); (iii) \textit{extrapolation} of the restricted macroscopic state variables over a larger time step $\Delta t \gg \delta t$; (iv) \textit{matching} the extrapolated state variables onto a probability distribution that perturbs the final distribution from (i) minimally.

Matching is the hardest step of the micro-macro acceleration algorithm. During matching, we build a new probability distribution that is consistent with the extrapolated state variables. This problem is often \textit{ill-posed}, since there can be many probability distributions consistent with a given set of macroscopic state variables. Therefore, it was proposed in~\cite{debrabant2017micro} to find the distribution that minimizes the divergence with respect to a prior distribution $P$. Such a prior is naturally available as the final distribution from the simulation step of the micro-macro acceleration method. In this work, we use matching introduced in \cite{debrabant2017micro,lelievre2018analysis} and based on minimizing the Kullback-Leibler divergence (also called relative entropy)
\[
\mathcal{D}( Q || P ) = \mathbb{E}_Q\!\left[ \ln \left(\rnder{Q}{P} \right) \right],
\]
over all distributions $Q$ that are consistent with the extrapolated states.

In this section, we first derive some explicit formulas for this matching procedure for linear slow fast SDEs (Section~\ref{sec:mM_match}), after which we present the complete micro-macro acceleration method in full mathematical detail (Section~\ref{sec:complete_mM}).

\subsection{Matching with the slow mean\label{sec:mM_match}}

In this paper, we are particularly interested in the matching procedure that reconstructs a full microscopic distribution based only on the slow mean. For more general matching operators, see~\cite{debrabant2017micro,lelievre2018analysis}. In this case, denoting by $\bar{\mu}^s$ the mean of the slow component, the matching reads
\begin{equation} \label{eq:slowmeanmatching}
\mathcal{M}(\bar{\mu}^s, P) = \underset{Q \in \mathscr{P}}{\text{arg min}} \  \mathcal{D}( Q || P ), \ \ \textrm{s.t.} \ \ \mathbb{E}_Q[\Pi^s] = \bar{\mu}^s.
\end{equation}
The distribution $\bar{Q}$ solving~\eqref{eq:slowmeanmatching} is always absolutely continuous with respect to the prior $P$ and its density has exponential shape given by
\begin{equation*}
    \rnder{\bar{Q}}{P}(x) = \exp\!\big(\bar{\lambda}^s \cdot x - A(\bar{\lambda}^s, P)\big),
\end{equation*}
where the normalization constant (log-partition function) is $A(\lambda^s, P) = \ln \mathbb{E}_P\left[\exp\left(\lambda^s \cdot \Pi^s \right) \right]$. The optimal Lagrange multipliers $\bar{\lambda}^s \in \mathbb{R}^{d_s}$ are unique and fulfill~\cite{lelievre2018analysis}
\begin{equation} \label{eq:lagrangemultipliersmean}
\nabla_{\lambda_s} A(\bar{\lambda}^s, P) = \bar{\mu}_s.
\end{equation}
In particular, when the prior $P$ has density $\pi_P$ with respect to the Lebesgue measure on $\R^d$, so does $\bar{Q}$ and its density reads
\[
\pi_Q(x) = \exp\left(\bar{\lambda}^s \cdot x - A\left(\bar{\lambda}^s, P\right) \right) \pi_P(x).
\]
Moreover, when the prior $P$ is Gaussian, there is a closed expression for $\bar{Q}$, as will become clear in Lemma \ref{lem:linearmatchingmean} (Section~\ref{sec:conv_lem}).  When $P$ is not Gaussian, we need to resort to numerical methods to solve~\eqref{eq:lagrangemultipliersmean} for the Lagrange multipliers, see, e.g.,~\cite{debrabant2017micro}.

The following result connects matching of the full microscopic distribution $P$ with a given slow mean $\bar{\mu}^s$ to the corresponding matching procedure using the slow marginal $P^s$ of the microscopic distribution as prior.
\begin{proposition} \label{pro:faststates}
Let $\bar{Q}=\match(\bar{\mu}^s, P)$ be the solution to \eqref{eq:slowmeanmatching} and $\bar{Q}^s=\match(\oline{\mu}^s, P^s)$ be the solution to the matching of slow prior marginal $P^s$. Then, the matching densities satisfy
\begin{equation*}
    \rnder{\bar{Q}}{P}(y,z) = \rnder{\bar{Q}^s}{P^s}(y),\quad y\in\R^{d_s},z\in\R^{d_f},
\end{equation*}
and, in particular, all slow observables of $\bar{Q}$ equal the corresponding observables of $\bar{Q}^s$. Moreover, for any function $f$ on $\R^{d_f}$, the fast observable of $\bar{Q}$ generated by $f$ is given as
\begin{equation}\label{eq:match_fast_mean}
    \Exp[\bar{Q}][f(Z)] = \int_{\R^{d_s}}\Exp[P][f(Z)|Y=y]\,\bar{Q}^s(\der{y}).
\end{equation}
\end{proposition}

\begin{proof}
To see the first identity, consider the log-partition function $A(\la,P)$. Employing the decomposition of $P$ into its marginal and conditional~\cite[Thm.~10.2.1]{Dudley2002}, we compute
\begin{gather*}\label{eq:logpart_marg}
\begin{aligned}
A(\la, P)
&= \ln\int_{\R^d}\exp\!\big(\la\sdot\Pi^sx\big)\,P(\der{x})\\[.2em]
&= \ln\!\Big\{\int_{\R^{d_s}}\exp\!\big(\la\sdot y\big)\int_{\R^{d_f}}P^{f|s}(\der{z}|y)\,P^s(\der{y})\Big\}\\[.2em]
&= \ln\Exp[P^s]\big[\exp(\la\sdot\Pi^s)] = A(\la, P^s),
\end{aligned}
\end{gather*}
where $A(\la,P^s)$ is the log-partition function of the marginal prior $P^s$. Since both log-partition functions agree, the vector of Lagrange multipliers $\bar{\la}\in\R^{d_s}$ of $\bar{Q}$ corresponds exactly to the one of $\bar{Q}^s$. Therefore, we can write
$
\rnder{\bar{Q}}{P}(y,z) = \exp\big(\oline{\la}\sdot y - A(\bar{\la},P^s)\big),
$
which proves the first identity.

To unwrap the fast mean of $\bar{Q}$, let us first write
\begin{equation*}
    \Exp[\bar{Q}][f(Z)] = \int_{\R^{d}}f(z)\,\bar{Q}(\der{y},\der{z}) = \int_{\R^{d}}f(z)\rnder{\bar{Q}}{P}(y,z)\,P(\der{y},\der{z}).
\end{equation*}
Using the identity between the densities and~\cite[Thm.~10.2.1]{Dudley2002} once more, we arrive at
\begin{gather*}
\begin{aligned}
    \Exp[\bar{Q}][f(Z)]
    &= \int_{\R^{d_s}}\int_{\R^{d_f}}f(z)\rnder{\bar{Q}^s}{P^s}(y)\,P^{f|s}(\der{z}|y)\,P^s(\der{y})\\
    &= \int_{\R^{d_s}}\Big[\int_{\R^{d_f}}f(z)\,P^{f|s}(\der{z}|y)\Big]\rnder{\bar{Q}^s}{P^s}(y)P^s(\der{y}),
\end{aligned}
\end{gather*}
which concludes the proof of~\eqref{eq:match_fast_mean}.
\end{proof}

\subsection{The complete micro-macro acceleration method}\label{sec:complete_mM}

In this section, we describe the four steps of the micro-macro acceleration algorithm in detail. 
We first present the time discretization of the linear SDE~\eqref{eq:linearsde}. We further introduce the restriction operator together with linear extrapolation of the macroscopic state variables. Finally, we use the matching operators as discussed in the previous section.

Let $P_n$ be the probability distribution at time $t_n = n\Delta t$. The micro-macro acceleration algorithm advances the distribution $P_n$ to a distribution $P_{n+1}$ at time $t_{n+1}=t_n+\Delta t$ in four stages:

\paragraph{Step 1: Microscopic time integration} In the first step, we perform a simulation of \eqref{eq:linearsde} over a time window of size $\delta \tau$. The computational cost of time propagation is usually high and we choose $\delta \tau$ to be of the order of the stiffest part of \eqref{eq:linearsde}. In practice, we usually take $K$ time steps of size $\delta t$, such that $\delta \tau = K \delta t$. In this text, we use the Euler-Maruyama scheme to discretize \eqref{eq:linearsde}, reading
\begin{equation} \label{eq:eulermaruyama}
    X_{n,k+1} = \left(I + A \delta t \right) X_{n,k} + \sqrt{\delta t}B \delta W_{n,k},
\end{equation}
for $k=1,\dots,K$, and where $\delta W_{n,k-1}$ are Brownian increments. The random variables $X_{n,k}, \ k = 0,\dots, K$ have probability distributions $P_{n,k}$ and we denote the initial distribution as $P_{n,0} = P_n$.

\paragraph{Step 2: Restriction} Second, to transition from the full microscopic description to the macroscopic state variables, we compute the mean of the slow component of the process \eqref{eq:linearsde}, reading $\mathbb{E}_P\left[ \Pi^s \right]$. We restrict the slow mean at every microscopic time step, generating a sequence of values of the slow mean:
\[
\mu^s_{n,k} = \mathbb{E}[\Pi^s P_{n,k}], \ \ k = 0,\dots,K.
\]

\paragraph{Step 3: Extrapolation}
In the third step, we perform time integration on the macroscopic level over a time interval of size $\Delta t$. Given the slow means $\mu^s_{n,k}, \ k=0,\dots,K$ at times $t_n + k\delta t$ from the previous step, we compute the slow mean $\mu^s_{n+1}$ at time $t_{n+1} = t_n + \Delta t$ by linear extrapolation
\begin{equation} \label{eq:linearextrapolation}
\mu^s_{n+1} = \mu^s_n + \frac{\Delta t}{K \delta t} \left(\mu^s_{n,K} - \mu^s_n \right).
\end{equation}
Note that we only use the state variables at time $t_n$ and $t_n + K\delta t$.

\paragraph{Step 4: Matching} Finally, we construct a new probability distribution that is consistent with $\mu^s_{n+1}$. To this end, we employ the matching operator \eqref{eq:slowmeanmatching} and define
\[
P_{n+1} = \mathcal{M} \left(\mu^s_{n+1}, P_{n,K} \right),
\]
to obtain a new microscopic distribution $P_{n+1}$ at time $t_{n+1} = t_n+ \Delta t$. The prior distribution $P_{n,K}$ is the final distribution computed during Step~1.

\section{Convergence of the micro-macro acceleration method with slow mean extrapolation\label{sec:convergence}}

In this section, we prove that the micro-macro acceleration method of Section~\ref{sec:mM_acc} converges to the exact dynamics of the linear SDE \eqref{eq:linearsde}, when only extrapolating the mean of the slow process and when the initial condition is Gaussian (Theorem~\ref{thm:linearslowconvergence} in Section~\ref{sec:thm1}). Before proceeding to the proof, we need an intermediate result, that explicitly describes the evolution of the mean and variance of the full microscopic system under the micro-macro acceleration method. This intermediate result is the subject of Section~\ref{sec:conv_lem}.  Theorem~\ref{thm:linearslowconvergence} differs from the main convergence result in \cite{lelievre2018analysis}, as the latter requires a hierarchy of macroscopic state variables to form a complete description of the density it represents. The slow mean by itself never forms such a complete description of the underlying density. An extension of Theorem \ref{thm:linearslowconvergence} to non-linear SDEs or non-Gaussian initial conditions is highly non-trivial.

\subsection{An iterative formula for slow mean-only extrapolation\label{sec:conv_lem}}
The proof of the convergence result in Theorem~\ref{thm:linearslowconvergence} relies on an iterative formula that describes how the complete mean and variance propagate through one step of the micro-macro acceleration scheme. As the micro-macro method preserves the Gaussianity of the initial condition, which we show below, the knowledge of the mean and variance suffices to control the distribution throughout the whole simulation. The derivation here assumes only one Euler-Maruyama inner step of size $\delta t$ for simplicity, but can easily be extended to $K$ inner steps.

We start with a lemma, proven also in \cite{debrabant2018study} but by different means, that gives the matched distribution when the prior is Gaussian and we only match with the slow mean.

\begin{lemma} \label{lem:linearmatchingmean}
	Suppose $P$ is the prior Gaussian distribution with mean $\mu$ and covariance matrix $\Sigma$,
	\[
		\mu = \begin{bmatrix}	\mu^s\\ \mu^f \end{bmatrix}, \ \ \Sigma = \begin{bmatrix} \Sigma^s & C \\ C^T & \Sigma^f \end{bmatrix}.
	\] 
	 The distribution $\bar{Q}=\match(\bar{\mu}^s,P)$, which solves~\eqref{eq:slowmeanmatching}, is also Gaussian with the same variance and mean $\bar{\mu} = [\bar{\mu}^s, \bar{\mu}^f]^T$ where $\bar{\mu}^f = \mu^f + C^T \left( \Sigma^s \right)^{-1}(\bar{\mu}^s - \mu^s)$.
\end{lemma}
\begin{proof}
Let $\bar{Q}^s = \match(\bar{\mu}^s, P^s)$, where $P^s$ is the slow marginal of $P$. Since $P^s=\nd(\mu^s,\Sigma^s)$, by the Gaussianity of $P$, a standard result for Kullback-Leibler minimization states that $\bar{Q}^s$ is also Gaussian with the mean $\bar{\mu}^s$ and the same variance $\Sigma^s$. That $\bar{Q}$ is Gaussian and its slow variance equals $\Sigma^s$ follows directly from the expression connecting the Radon-Nikodym derivatives in Proposition~\ref{pro:faststates}.

To compute the fast mean $\bar{\mu}^f$ and variance $\bar{\Sigma}^f$ of $\bar{Q}$, we use the second part of Proposition~\ref{pro:faststates}. Focusing on $\bar{\mu}^f$ first, employing formula~\eqref{eq:match_fast_mean} with $f(z)=z$ and a well-known expression for the conditional mean for Gaussian distributions, we can write the matched fast mean as
\begin{equation*}
\bar{\mu}^f 
= \int_{\mathbb{R}^{d_s}} \big\{ \mu^f + C^T (\Sigma^s)^{-1}(y - \mu^s) \big\} \, \bar{Q}^s(dy)
= \mu^f + C^T (\Sigma^s)^{-1}(\bar{\mu}^s - \mu^s).
\end{equation*}
Similarly, we can express the fast matched variance by choosing $f(z) = (z-\bar{\mu}^f)(z-\bar{\mu}^f)^T$ in~\eqref{eq:match_fast_mean}
\begin{equation*}
\bar{\Sigma}^f = 
\int_{\mathbb{R}^{d_s}} \mathbb{E}_P[(Z-\bar{\mu}^f)(Z-\bar{\mu}^f)^T| Y=y] \ \bar{Q}^s(dy).
\end{equation*}
By adding and subtracting the fast conditional mean $\mu^{f|s}(y) = \mu^f + C^T \left(\Sigma^s \right)^{-1}(y-\mu^s)$ to each $Z-\bar{\mu}^f$, we get
\begin{align*}
\centering
\bar{\Sigma}^f 
&=\int_{\mathbb{R}^{d_s}} \big\{ \mathbb{E}_P[(Z-\mu^{f|s}(y))(Z-\mu^{f|s}(y))^T | Y = y] + (\mu^{f|s}(y)- \bar{\mu}^f)(\mu^{f|s}(y)- \bar{\mu}^f)^T\big\}\, \bar{Q}^s(dy).
\end{align*}
The first summand under the integral represents the fast conditional variance of $P$ and, since $P$ is Gaussian, it is independent of $y$ and equals $\Sigma^f - C^T \left(\Sigma^s \right)^{-1} C$. The second summand can be expanded using the expressions for $\mu^{f|s}(y)$ and $\bar{\mu}^f$. Writing it out, we obtain
\begin{align*}
\centering
\overbar{\Sigma}^f &= \Sigma^f - C^T \left(\Sigma^s \right)^{-1} C + \int_{\mathbb{R}^{d_s}} C^T \left(\Sigma^s \right)^{-1} (y-\bar{\mu}^s) (y-\bar{\mu}^s)^T \left (\Sigma^s \right)^{-1} C \, \bar{Q}^s(dy) \\
&= \Sigma^f - C^T \left(\Sigma^s \right)^{-1} C + C^T \left(\Sigma^s \right)^{-1} \Sigma^s \left( \Sigma^s \right)^{-1} C = \Sigma^f,
\end{align*}
where we used the fact that the slow variance of $\bar{Q}^s$ equals $\Sigma^s$.
\end{proof}

With Lemma~1, we are armed to obtain a closed expression for the time-discrete evolution of the mean of the slow-fast SDE~\eqref{eq:linearsde}, as generated by the micro-macro acceleration method of Section~\ref{sec:mM_acc}. Suppose that the drift matrix $A$ is given in block form
\[
	A = \begin{bmatrix} A^s & V \\ W & A^f \end{bmatrix}.
\]
and that at time $t_n = n \Delta t$ the distribution is Gaussian with mean $\mu_n$, and covariance $\Sigma_n$ with
\[
\mu_n = \begin{bmatrix}\mu^s_n\\ \mu^f_n \end{bmatrix}.
\]
We can then write each of the four steps of the algorithm in explicit form. First, we consider the microscopic simulation step. During one Euler-Maruyama step, a Monte Carlo particle $X_n$ is propagated as
\[
X_{n,1} = (I+A\delta t)X_n + \sqrt{\delta t}B \xi_n, \ \ \xi_n \sim \mathcal{N}(0,1),
\]
and the distribution of $X_{n,1}$ is also Gaussian.  Next, we perform the restriction step. By taking expectations, the mean $\mu_{n,1}$ and the covariance $\Sigma_{n,1}$ after the Euler-Maruyama step read
\begin{equation}\label{eq:em_meanvar_rec}
\mu_{n,1} = (I+A\delta t)\mu_n, \ \ \Sigma_{n,1} = (I+\delta t A)\Sigma_n(I+\delta t A)^T + \delta t BB^T,
\end{equation}
after which we extrapolate the slow mean as
\begin{equation} \label{eq:slowmeanpropagation}
\begin{split}
\mu^s_{n+1} &= \mu^s_n + \frac{\Delta t}{\delta t} \left ( \mu_{n,1}^s - \mu_n^s \right) \\
	&=  \mu^s_n + \frac{\Delta t}{\delta t} \left ( \mu^s_n + \delta t A^s \mu^s_n + \delta tV \mu^f_n - \mu^s_n \right) \\
	&= (I^s + \Delta t A^s) \mu_n^s + \Delta t V \mu^f_n.
\end{split}
\end{equation}

Now that we have the extrapolated slow mean, we can use Lemma~\ref{lem:linearmatchingmean} to explicitly obtain the result of matching. According to Lemma~\ref{lem:linearmatchingmean}, the matched distribution is also Gaussian when only extrapolating the slow mean. Furthermore, the covariance matrix  is not affected by matching, i.e., $\Sigma_{n+1}=\Sigma_{n,1}$, and the fast mean is given by
\begin{equation} \label{eq:fastmeanpropagation}
\begin{split}
\mu^f_{n+1} &= \mu^f_{n,1} + C^T_{n,1} \left(\Sigma_{n,1}^s \right)^{-1} \left(\mu^s_{n+1} - \mu^s_{n,1} \right) \\
	&= \delta t W \mu^s_n + (I^f+\delta t A^f)\mu^f_n + C^T_{n,1} \left(\Sigma_{n,1}^s \right)^{-1} \left((\Delta t - \delta t) A^s \mu^s_n + (\Delta t-\delta t)V \mu^f_n \right) \\
	&= \left[ \delta t W + (\Delta t-\delta t) C^T_{n,1} \left(\Sigma_{n,1}^s \right)^{-1}  A^s \right] \mu^s_n + \left[ I^f + \delta t A^f + (\Delta t-\delta t) C^T_{n,1} \left(\Sigma_{n,1}^s\right)^{-1} V \right] \mu^f_n.
\end{split}
\end{equation}
Bundling the propagation of the slow \eqref{eq:slowmeanpropagation} and fast matched mean \eqref{eq:fastmeanpropagation} in one vector $\mu_{n+1}$ gives
\begin{equation} \label{eq:meanpropagation}
\begin{bmatrix} \mu^s_{n+1} \\ \mu^f_{n+1} \end{bmatrix}= \begin{bmatrix} I^s + \Delta t A^s &  \Delta t V \\  \delta t W + (\Delta t-\delta t) C^T_{n,1} \left(\Sigma_{n,1}^s \right)^{-1}  A^s & I^f + \delta t A^f + (\Delta t-\delta t) C^T_{n,1} \left(\Sigma_{n,1}^s\right)^{-1} V  \end{bmatrix} \begin{bmatrix} \mu^s_{n} \\ \mu^f_{n} \end{bmatrix}.
\end{equation}
To conclude, the time-discrete evolution of the mean of the slow-fast SDE~\eqref{eq:linearsde}, as generated by the micro-macro acceleration method is given by the time-dependent linear system~\eqref{eq:meanpropagation}, with initial condition equal to the mean of the initial distribution of \eqref{eq:linearsde}.

\subsection{Convergence theorem \label{sec:thm1}}
All elements are now in place to prove convergence of the micro-macro acceleration method of Section~\ref{sec:mM_acc} that only extrapolates the slow mean of the process. The proof makes use of the iterative formula above, and holds for general linear SDEs with additive noise.

\begin{theorem} \label{thm:linearslowconvergence}
	Given a linear SDE with a Gaussian initial distribution, consider the micro-macro acceleration algorithm with relative-entropy matching and slow-mean extrapolation. Also, fix an end time $T > 0$. Denote by $P_T$ the exact distribution of the linear SDE at time $T$, and by $P_{n_{\Delta t}(T)}$ the distribution obtained using $n_{\Delta t}(T)$ steps of the micro-macro acceleration scheme of Section~\ref{sec:mM_acc} with $K=1$, where $n_{\Delta t}(T) = \lfloor T/\Delta t\rfloor$. Then,
	\begin{equation} \label{eq:proofobjective}
	\underset{\delta t \to 0}{\lim}	\ \underset{\Delta t \to \delta t}{\lim} \	
	\mathcal{D}(P_{n_{\Delta t}(T)} \  || \ P_T) = 0.
	\end{equation}
	As a consequence, the distributions $P_{n_{\Delta t}(T)}$ obtained by micro-macro acceleration at time $T$ converge in total variation to the exact microscopic distribution $P_T$ in the same limits.
\end{theorem}
\begin{proof}
	Since the initial condition is Gaussian, all intermediate distributions of the exact solution, the Euler-Maruyama method, and micro-macro acceleration are Gaussian too. By a standard expression for the Kullback-Leibler divergence between two Gaussian distributions \cite{hershey2007approximating}, the divergence in~\eqref{eq:proofobjective} becomes
	\begin{equation*}
	\begin{aligned}
	\mathcal{D}(P_{n_{\Delta t}(T)}\,||\, P_T) &= \frac{1}{2} \Big( \ln\frac{\left|\Sigma_T\right|}{\left|\Sigma_{n_{\Delta t}(T)}\right|} - d + \text{Tr}\big(\Sigma_T^{-1} \Sigma_{n_{\Delta t}(T)}\big) + (\mu_T - \mu_{n_{\Delta t}(T)})^T \Sigma_T^{-1}(\mu_T - \mu_{n_{\Delta t}(T)})\Big),
	\end{aligned}
	\end{equation*}
	where $\mu_T$ and $\Sigma_T$ are the mean and variance of $X_T$, and $\mu_{n_{\Delta t}(T)}$ and $\Sigma_{n_{\Delta t}(T)}$ are the mean and variance of $X_{n_{\De t}(T)}$. 
	
	First, we fix $\delta t \leq \Delta t$ and let $\Delta t$ decrease to $\delta t$. If we perform back-substitution in equation \eqref{eq:meanpropagation} to write the mean at time $T$ as a function of the initial mean vector, we obtain a product of $n_{\Delta t}(T)$ different matrices. The number of matrices increases to $n_{\delta t}(T)$ as $\Delta t$ decreases to $\delta t$, but there always remain a finite number of matrices because $\delta t > 0$. The contribution of the largest off-diagonal term in \eqref{eq:meanpropagation} also reduces to zero and as a result, the mean vector $\mu_{n_{\Delta t}(T)}$ approaches the respective mean $\mu_{n_{\delta t}(T)}$ of the Euler-Maruyama scheme. We obtain
	\begin{equation*} \label{eq:iterativewrittenout}
	\begin{aligned}
	\lim\limits_{\Delta t \to \delta t} \begin{bmatrix} \mu^s_{n(\Delta t)} \\ \mu^f_{n(\Delta t)}\end{bmatrix} = \begin{bmatrix} I^s + \delta t A^s &  \delta t V \\  \delta t W  & I^f + \delta t A^f  \end{bmatrix} ^{n(\delta t)} \begin{bmatrix} \mu^s_0 \\ \mu^f_0 \end{bmatrix},
	\end{aligned}
	\end{equation*}
	where $\mu_0^s$ and $\mu_0^f$ are the slow and fast mean respectively of the initial condition. Similarly, since the variance stays constant during matching, $\Sigma_{n_{\De t}(T)}$ converges to $\Sigma_{n_{\de t}(T)}$, the variance of the Euler-Maruyama scheme. Hence, the limit \eqref{eq:proofobjective} reduces to 
	\begin{equation*}
	\begin{aligned}
	\underset{\delta t \to 0} {\lim}\  
	\frac{1}{2} \Big( \ln\frac{\left|\Sigma_T\right|}{\left|\Sigma_{n_{\delta t}(T)}\right|} - d + \text{Tr}\big(\Sigma_T^{-1} \Sigma_{n_{\delta t}(T)}\big)  
	+ (\mu_T - \mu_{n_{\delta t}(T)})^T \Sigma_T^{-1}(\mu_T - \mu_{n_{\delta t}(T)})\Big),
	\end{aligned}
	\end{equation*}
	which is the relative entropy between the Euler-Maruyama scheme and the exact solution at time $T$. 
	This expression converges to zero because, as $\delta t$ decreases to zero, the mean and variance of the Euler-Maruyama method converge to their respective values of the exact solution at time $T$. Finally, by Pinsker's inequality, $P_{n_{\De t}(T)}$ converges to $P_T$ in total variation.
\end{proof}

\noindent
Theorem \ref{thm:linearslowconvergence} might be surprising, since it does not require the number of macroscopic state variables to increase to infinity as the extrapolation time step $\Delta t$ decreases; using only the slow mean as a macroscopic state variable is sufficient. However, the result only holds for linear SDEs with Gaussian initial conditions, as the proof relies heavily on iteration \eqref{eq:meanpropagation}. At the moment, no proof exists on convergence for non-linear SDE with any (fixed) finite number of macroscopic state variables.

\section{Stability of micro-macro acceleration with initial condition with Gaussian tails}\label{sec:stability}

In this section, we study the stability of micro-macro acceleration when applied to~\eqref{eq:linearsde}, i.e., the convergence of the laws it generates, in the limit as the number of extrapolation with fixed step size $\Delta t$ goes to infinity, to the invariant distribution of the underlying Euler-Maruyama scheme. When $\de t$ denotes the microscopic step, the invariant distribution of the Euler-Maruyama scheme is the zero-mean normal distribution $\nd_{0,V_\infty^{\de t}}$, with variance
\begin{equation}\label{eq:invEM_var}
    V_\infty^{\de t} = \de t \sum_{j=0}^{\infty}(I+\de t A)^j B\tp{B}(I+\de t\tp{A})^j,
\end{equation}
as can be seen by repeatedly applying the recursion for the mean and variance in~\eqref{eq:em_meanvar_rec}.

The stability question was analyzed in~\cite{debrabant2018study} in the Gaussian setting, where the distributions that are generated by the micro-macro acceleration can be computed explicitly. The main result there focuses on a simple diagonal case and reads:
\begin{theorem}\label{thm:stabdiag}
When applying the micro-macro acceleration method to the linear SDE~\eqref{eq:linearsde} with block-diagonal drift matrix
\begin{equation*}
    A = \begin{bmatrix} A^s & 0 \\ 0 & A^f \end{bmatrix}
\end{equation*}
the mean $\mu_n$ and the covariance matrix $\Si_n$ of the resulting Gaussian law at the $n$th step satisfy
\begin{equation*}
    \lim_{n\to\infty}\mu_n=0,\quad\lim_{n\to\infty}\Si_n=V_\infty^{\de t},
\end{equation*}
whenever
\begin{equation}\label{eq:sprbound}
    \spr{I^s + \De tA^s}<1
    \quad\text{and}\quad
    \spr{I^s + \de tA^f}<1.
\end{equation}
\end{theorem}
Here, $\spr{\cdot}$ denotes the \emph{spectral radius} of a matrix. Condition~\eqref{eq:sprbound} is necessary to stabilize the extrapolation of the slow mean $\mu_n^s$, by bounding the values of $\De t$, and to stabilize the Euler-Maryuama stage, by bounding the values of $\de t$.  For the proof of Theorem~\ref{thm:stabdiag}, it suffices to establish the proper asymptotic behaviour of the means $\mu_n$ and variances $\Sigma_n$ of Gaussian distributions generated by the micro-macro acceleration scheme.

In this section, we go beyond the Gaussian case and work within the larger class of probability measures that have Gaussian tails. As a consequence, we do not have explicit formulas for the distributions generated by the scheme and we have to control all  moments to show stability. Rather than trying to obtain stability bounds on $\De t$ directly, we concentrate on showing that convergence of $\mu^n$ to $0$, i.e., asymptotic stability of the mean, already yields  convergence of the distributions to $\nd_{0,V_\infty^{\de t}}$. The relation between these two notions was illuminated in the Gaussian setting in \cite{debrabant2018study}. 

The main stability result of this paper, Theorem~\ref{th:mM_mean_stab} in Section~\ref{sec:stab_result}, gives  weak convergence of distributions produced by the micro-macro acceleration method to $\nd_{0,V_\infty^{\de t}}$, as the number of extrapolation steps goes to infinity. To prove it, we  explore  the properties of \emph{cumulant generating functions} (CGFs) to produce a recursion formula for the laws generated by the method (Section~\ref{sec:cdf}).


\subsection{Micro-macro step in terms of the cumulant generating function\label{sec:cdf}}

Let us first define the cumulant generating function as follows.
\begin{definition}\label{dfn:cgf}
For any probability distribution $P\in\prob^d$, we define the \emph{cumulant generating function} of $P$
\begin{equation*}
\cgf_{P}(\vth) = \ln\Exp[P]\big[e^{\vth\sdot\Pi}\big],\quad \vth\in\Th_{P},
\end{equation*}
where $\Pi$ is the identity on $\R^d$, and the \emph{effective domain} reads $\Th_{P} = \{\vth\in\R^d:\ \cgf_{P}(\vth)<+\infty\}$.
\end{definition}
When $X\sim P$ we also write $\cgf_X$ instead of $\cgf_P$, and $\Th_X$ instead of $\Th_P$. We summarize the basic properties of CGFs in~\ref{se:cgf}.
\begin{example}\label{ex:cgf_nd}
If $X\sim\nd_{\mu,\Si}$, then
\begin{equation*}
\cgf_X(\vth) = \mu\sdot\vth + \frac{1}{2}\tp{\vth}\Si\vth.\qedhere
\end{equation*}
\end{example}

To effectively use the CGFs to describe the micro-macro acceleration procedure, we assume that for the initial random variable $X_0$ it holds $\Th_0\doteq\Th_{X_0}=\R^d$. In view of Proposition~\ref{pro:cgf_prop_match}, the CGF of the matched distribution results from shifting and translating the CGF of the prior based on the current value of the Lagrange multipliers. Having priors with full effective domain avoids the issue of falling outside the effective domain while shifting the CGF -- a clear sign that the matching is impossible.

\begin{remark}[What does $\Th_0=\R^d$ mean?]\label{re:cgf_fulldom}
Let us fix $\vth\in\R^d$ and $r>0$. According to the Chernoff's bound~\cite[p.~392]{cover2012elements}, we have
\begin{equation}\label{eq:chernoff}
\mbb{P}(\vth\sdot X_0\geq r) \leq e^{-sr + \cgf_{\vth\sdot X_0}(s)},
\end{equation}
for all $s\geq0$. From Proposition~\ref{pro:cgf_prop}(\ref{pro:cgf_prop_lintran}) in the Appendix, applied with $l=1$ and $M=\tp{\vth}$, we have $\cgf_{\vth\sdot X_0}(s) =\cgf_{X_0}(s\vth)$. Thus, taking logarithms on both sides of~\eqref{eq:chernoff}, we can equivalently write
\begin{equation*}
-\ln\mbb{P}(\vth\sdot X_0\geq r)\geq sr - \cgf_0(s\vth),
\end{equation*}
where we denote $\cgf_0=\cgf_{X_0}$. Since $\cgf_0(s\vth)$ is finite for all $s\geq0$, dividing by $r$ and taking the limit gives
\begin{equation*}
\lim_{r\to+\infty}\frac{-\ln\mbb{P}(\vth\sdot X_0\geq r)}{r} \geq s.
\end{equation*}
Therefore, because $s$ can be arbitrarily large, the log-tail function of $\vth\sdot X_0$ is superlinear at $+\infty$. The same holds at $-\infty$ by repeating the argument for $\mbb{P}(\vth\sdot X_0\leq -r)$.
\end{remark}

To simplify the notation, we again use only one micro step for each extrapolation in the micro-macro acceleration procedure. In the Proposition below, we consider the micro-macro acceleration method as applied to the linear slow-fast SDE~\eqref{eq:linearsde}.
\begin{proposition}
Let $n\geq1$ and assume that for a~random variable $X_{n-1}$ with cumulant generating function $\cgf_{X_{n-1}}=\cgf_{n-1}$ we have $\Th_{n-1}=\R^d$. Then, if $X_n$ is obtained from the micro-macro procedure with extrapolation of the (slow) $s$-marginal mean (as described in Section~\ref{sec:complete_mM}, with $K=1$), its CGF $\cgf_n=\cgf_{X_n}$ has effective domain $\Th_n=\R^d$ and satisfies
\begin{gather}\label{eq:mM_cum}
\begin{aligned}
\cgf_{n}(\vth)
&= \cgf_{n-1}\big((I+\de t\tp{A})(\vth+\la^s_n\oplus0^f)\big) - \cgf_{n-1}\big((\Id+\de t\tp{A})(\la^s_n\oplus0^f)\big)\\[0.2em]
&\hphantom{=\ }+ \frac{\de t}{2}\big[\tp{\vth}\!B\tp{B}\vth+ \tp{(\la^s_n\oplus0^f)}\!B\tp{B}\vth + \tp{\vth}\!B\tp{B}(\la^s_n\oplus0^f)\big],
\end{aligned}
\end{gather}
where $\la^s_n$ is a~vector of Lagrange multipliers.
\end{proposition}

\begin{proof}
According to Section 2, the law of $X_n$ is given by $\match(\mu^s_n,\mrm{Law}(X^{\de t}_{n-1,1}))$, where $\mu^s_n$ is extrapolated as in \eqref{eq:linearextrapolation} and $X^{\de t}_{n-1,1}$ is obtained from $X_{n-1}$ by one Euler-Maruyama step over $\de t$. Therefore, for the cumulant, Proposition~\ref{pro:cgf_prop_match} yields
\begin{equation}\label{eq:match_cum}
\cgf_{n}(\vth) = \cgf_{n-1,1}(\vth+\la^s_n\oplus0^f) - \cgf_{n-1,1}(\la^s_n\oplus0^f),
\end{equation}
where $\la^s_n$ is the Lagrange multiplier associated to the extrapolated marginal mean $\mu^s_n$.

In the current notation, the recursive formula~\eqref{eq:eulermaruyama} reads
\begin{equation*}
X^{\de t}_{n-1,1}=(\Id+\de tA)X_{n-1}+B\de W_{n,1}.
\end{equation*}
Because $X_{n-1}$ and $\de W_{n,1}$ are independent,  we obtain from Proposition~\ref{pro:cgf_prop}(\ref{pro:cgf_prop_indep}) that
\begin{equation*}
\cgf_{n-1,1} = \cgf_{(I+\de tA)X_{n-1}} + \cgf_{B\de W_{n,1}}.
\end{equation*}
The law of $B\de W_{n,1}$ is $\nd_{0,\de tB\!\tp{B}}$ thus, according to Example~\ref{ex:cgf_nd}, $\cgf_{B\de W_{n,1}}(\vth) = \de t/2\,\tp{\vth}\!B\tp{B}\vth$ and, using Proposition~\ref{pro:cgf_prop}(\ref{pro:cgf_prop_lintran}), we obtain
\begin{equation}\label{eq:euler_cum}
\cgf_{n-1,1}(\vth) = \cgf_{n-1}\big((\Id+\de t\tp{A})\vth\big) + \frac{\de t}{2}\tp{\vth}B\tp{B}\vth.
\end{equation}

Combining~\eqref{eq:match_cum} with~\eqref{eq:euler_cum} results in~\eqref{eq:mM_cum}. Since the effective domain of $X_{n-1}$ was equal to $\R^d$, the right-hand side of~\eqref{eq:mM_cum} is finite for all $\vth$. This implies $\Th_{n}=\R^d$.
\end{proof}

\subsection{Convergence to the equilibrium with stable mean extrapolation\label{sec:stab_result}}

The main result of this section depends on the following assumption on the initial random variable for the micro-macro acceleration method. Recall from Proposition~\ref{pro:cgf_prop}(\ref{pro:cgf_prop_tayl}) that the cumulant generating function is always analytic on the interior of its effective domain. Here and in what follows, for any constant $C\geq0$ and functions $f,g\from\R_{+}\to\R$ with $g$ positive in a~neighbourhood of infinity, $f(r)\otof Cg(r)$ means $\lim_{r\to+\infty}f(r)/g(r)=C$.

\begin{assumption}\label{as:cgf_der_asym}
The CGF $\cgf_0$, of an initial random variable $X_0$, has full effective domain (i.e., $\Th_{0}=\R^d$) and for every $\vth\in\R^d$ it satisfies
\begin{equation}\label{eq:cgf_der_asym}
\cgf_0(r\vth)\otof w_0(\vth)\,r^{2},\quad
\frac{\der}{\der{r}}\cgf_0(r\vth)\otof 2\,w_0(\vth)\,r,
\end{equation}
where $w_0\from\R^d\to(0,+\infty)$ is continuous and homogeneous of order $2$.
\end{assumption}

Assumption~\ref{as:cgf_der_asym} derives from the theory of regular variation. Let us first discuss its connection with the tails of random variables.

\begin{remark}[On log-quadratic tails]
The asymptotic relation $\cgf_0(r\vth)\otof w_0(\vth)\,r^{2}$ and Proposition~\ref{pro:cgf_prop}(\ref{pro:cgf_prop_lintran}) leads to $\cgf_{\vth\sdot X_0}(r)\otof w_0(\vth)\,r^{2}$. By the Kasahara-Tauberian Theorem~\cite[Thm.~4.12.7]{bingham1989regular}, the last relation is equivalent to the property
\begin{equation*}
-\ln\mbb{P}(\vth\sdot X_0\geq r)\otof-\ln\mbb{P}(\vth\sdot X_0\leq-r)\otof w_0(\vth)^{-1}r^2,\qquad\text{as}\ r\goesto+\infty.
\end{equation*}
That is, for every $\vth\in\R^d$, the random variable $\vth\sdot X_0$ has \emph{regularly varying log-quadratic} tail decay, a feature shared by all Gaussian laws. Therefore, Assumption~\ref{as:cgf_der_asym} sharpens the superlinear behaviour of tail functions that resulted from assuming $\Th_0=\R^d$, see Remark~\ref{re:cgf_fulldom}.
The inclusion of derivatives in~\eqref{eq:cgf_der_asym} is related to the notion of higher-order regular variation(compare~\cite{granata2016theory} and~\cite[p.~44]{bingham1989regular}).

\end{remark}


In the proof of Theorem~\ref{th:mM_mean_stab} below, we employ the following technical result.

\begin{lemma}\label{lem:cgf_asym_shift}
Let $\cgf$ be continuously differentiable function on $\R^{d}$  that for every $\vth\in\R^{d}$ satisfies $\cgf(r\vth)\otof w(\vth)\,r^{2}$ and 
\begin{equation}\label{eq:grad_cgf_linbound}
    \sup_{\vth\neq0}\frac{\|\grad\cgf(\vth)\|}{\|\vth\|} < +\infty.
\end{equation}
Then, for every $\vth,\vth_0\in\R^{d}$, it holds $\cgf(r\vth+\vth_0)\otof 2\,w(\vth)r^2$.
\end{lemma}

\begin{proof}
Express the ratio $\cgf(r\vth+\vth_0)/r^2$ as the following sum
\begin{equation*}
\frac{\cgf(r\vth+\vth_0)-\cgf(r\vth)}{r^2} + \frac{\cgf(r\vth)}{r^2}.
\end{equation*}
The second summand converges to $2w(\vth)$ as $r\goesto+\infty$ from the asymptotic property of $\cgf$. That the first fraction disappears can be seen by applying the mean value inequality
\begin{equation*}
    \frac{|\cgf(r\vth+\vth_0)-\cgf(r\vth)|}{r^2} \leq \frac{\|\grad\cgf(r\vth+r'\vth_0)\|\|\vth_0\|}{r^2},
\end{equation*}
where $r'\in[0,1]$. Since
\begin{equation*}
    \frac{\|\grad\cgf(r\vth+r'\vth_0)\|}{r} = \frac{\|\grad\cgf(r\vth+r'\vth_0)\|}{\|r\vth+r'\vth_0\|}\cdot \big\|\vth+\frac{r'}{r}\vth_0\big\|
\end{equation*}
stays bounded as $r\to+\infty$, due to~\eqref{eq:grad_cgf_linbound}, the right-hand side of the previous inequality converges to zero in this limit.
\end{proof}


\begin{theorem}\label{th:mM_mean_stab}
Suppose that the CGF $\cgf_0$ of an initial random variable $X_0$ satisfies Assumption~\ref{as:cgf_der_asym}, both $\cgf_0(\vth)$, and $\grad\cgf_0(\vth)\cdot\vth$ fulfill~\eqref{eq:grad_cgf_linbound}, and $\cgf_0$ has all derivatives of order $3$ or higher bounded. If $\spr{\Id+\de tA}<1$, and the mean $\mu_n$ of random variables $X_n$ obtained from the micro-macro acceleration method is stable, that is $\lim_{n\to+\infty}\mu_n=0$, the laws of $X_n$ converge weakly to $\nd_{0,V_\infty^{\de t}}$.
\end{theorem}
The stability of the mean $\mu_n$ is closely related to the stability bounds on $\De t$ that guarantee that the slow marginal mean $\mu_n^s$ satisfies $\lim_{n\to\infty}\mu_n^s=0$. The behaviour of fast marginal $\mu_n^f$ is influenced at each step by both extrapolation over $\De t$ and the matching. In the Gaussian setting, Lemma~\ref{lem:linearmatchingmean} provides the exact formula for $\mu_n^f$ in terms of $\mu_n^s$ and the first two moments of distributions produced by the Euler-Maryuama step. Using this formula, we can obtain explicit stability bounds on $\De t$ and $\de t$, exemplified in~\eqref{eq:sprbound}. In the non-Gaussian setting, we can only say that the influence of matching on $\mu_n^f$ is encoded in the nonlinear procedure to obtain Lagrange multipliers $\la_n$. Having no explicit formulas for the multipliers we work under the assumption that the full mean is stable when using the micro-macro acceleration method. 

\begin{proof}
To establish stability, we employ the recursive relation~\eqref{eq:mM_cum}. However, due to the presence of the Lagrange multipliers $\la^s_n$ in the argument of $\cgf_n$ and in the additional last term of~\eqref{eq:mM_cum}, we cannot immediately pass to the limit as $n$ goes to $+\infty$. These multipliers exist for all $n$, a~consequence of $\Th_n=\R^d$, but we do not have any \emph{a~priori} estimates that would allow to control $\la^s_n$ as $n$ increases. Therefore, our strategy in proving the convergence is to look at the recurrences for the tails of $\cgf_n$, which do not contain $\la^s_n$ any more. Before we look at $\cgf_n$ itself, let us use the boundedness of its higher order derivatives to show the convergence of higher cumulants. 

First note that the last term in formula~\eqref{eq:mM_cum} is of second order in $\vth$, so it disappears after differentiating this identity three times. More precisely,  we have, for any $j\geq 3$,
\begin{equation*}
\sder^j\cgf_n(\vth)[\vth'] = \sder^j\cgf_{n-1}\big((I+\de t \tp{A})(\vth+\la^s_n\oplus0^f)\big)\big[(I+\de t\tp{A})\vth'\big],
\end{equation*}
where we treat the $j$th order derivative $\sder^j\cgf_n(\vth)$ at $\vth\in\R^{d}$ as a symmetric multilinear mapping on $(\R^d)^j$ and denote $\sder^j\cgf_n(\vth)[\vth']\doteq\sder^j\cgf_n(\vth)[\vth',\dotsc,\vth']$. By back-substituting we obtain
\begin{equation*}
\sder^j\cgf_n(\vth)[\vth'] = \sder^j\cgf_{0}\Big((I+\de t \tp{A})^n\vth+\sum_{k=1}^{n}(I+\de t \tp{A})^k(\la^s_k\oplus0^f)\Big)\big[(I+\de t\tp{A})^n\vth'\big].
\end{equation*}
Since $\sder^j\cgf_0$ is bounded on $\R^d$ for all $j\geq3$, we can estimate from the above relation that
\begin{equation*}
\|\sder^j\cgf_n(\vth)\|_{\mrm{mult}}\leq \max_{\vth'\in\R^d}\|\sder^j\cgf_0(\vth')\|_{\mrm{mult}}\|I+\de t\tp{A}\|^{jn},
\end{equation*}
where $\|\sder^j\cgf_n(\vth)\|_{\mrm{mult}} \doteq\sup|\sder^j\cgf_n(\vth)(\vth^1,\dotsc,\vth^j)|/\|\vth^1\|\dotsm\|\vth^j\|$ and the supremum is taken over all $\vth^i\neq0$. Since $\spr{\Id+\de tA}<1$, this demonstrates that $\sder^j\cgf_n(\vth)$ converges to zero as $n\goesto+\infty$, uniformly in $\vth$. In particular, $\lim_{n\to+\infty}\sder^j\cgf_n(0)=0$ for all $j\geq3$.

Let us now return to~\eqref{eq:mM_cum}. We show by induction that for all $\vth\in\R^d$ the function $r\mapsto\cgf_{n}(r\vth)$ is asymptotically quadratic in $r$ and find the recursion for the corresponding constants $w_n(\vth)$. Suppose that $\cgf_{n-1}(r\vth)\otof w_{n-1}(\vth)r^2$ with some function $w_{n-1}$.
Note that Assumption~\ref{as:cgf_der_asym} guarantees that this holds for $\cgf_0$. The recursive relation~\eqref{eq:mM_cum} gives
\begin{gather}\label{eq:cgf_rec}
\begin{aligned}
\frac{\cgf_{n}(r\vth)}{r^2}
&= \frac{\cgf_{n-1}\big(r(I+\de t\tp{A})\vth + (I+\de t\tp{A})(\la^s_n\oplus0^f)\big)}{r^2} - \frac{\cgf_{n-1}\big((I+\de t\tp{A})(\la^n_n\oplus0^f)}{r^2}\\[0.2em]
&\hphantom{=\ }+ \frac{\de t}{2}\frac{r^2\tp{\vth}B\tp{B}\vth + r\tp{(\la^s_n\oplus0^f)}B\tp{B}\vth + r\tp{\vth}B\tp{B}(\la^s_n\oplus0)}{r^2}.
\end{aligned}
\end{gather}
Note that simultaneously, by differentiating~\eqref{eq:mM_cum}, the gradient of $\cgf_n$ satisfies
\begin{equation}\label{eq:cgf_grad_rec}
\grad\cgf_n(\vth) = (I+\de tA)\grad{\cgf_{n-1}}\big((I+\de t\tp{A})(\vth+\la^s_n\oplus0^f)\big) + \de t\tp{(\vth+\la^s_n\oplus0^f)}\!B\tp{B},
\end{equation}
thus property~\eqref{eq:grad_cgf_linbound} propagates from $\cgf_0$ throughout all $\cgf_{n}$. Therefore, we can use Lemma~\ref{lem:cgf_asym_shift}, together with the inductive assumption, to conclude that the limit as $r$ goes to $+\infty$ on the right hand side exists and is equal to $2w_{n-1}\big((I+\de t\tp{A})\vth\big)+ \de t\tp{\vth}B\tp{B}\vth$. Denoting by $2w_n(\vth)$ the limit of the left hand side, we have the  recursion $w_n(\vth) = w_{n-1}\big((I+\de t\tp{A})\vth\big)+ (\de t/2)\tp{\vth}\!B\tp{B}\!\vth$. By back-substituting, we obtain
\begin{equation}\label{eq:cgf_slope_n}
w_n(\vth) = w\big((I+\de t\tp{A})^n\vth\big) + \frac{\de t}{2}\tp{\vth}\Big[\sum_{j=0}^{n-1}(I+\de t A)^jB\tp{B}(I+\de t\tp{A})^j\Big]\vth.
\end{equation}

Note also that, by a similar reasoning as for $\cgf_n$, but using~\eqref{eq:cgf_grad_rec} this time, we can inductively establish the relation $\der/\!\der{r}\cgf_n(r\vth)\otof 2w_n(\vth)r$. In particular, after differentiating, this relation yields
\begin{equation}\label{eq:grad_cgf_asym}
    \lim_{r\to+\infty}\frac{\grad\cgf_n(r\vth)\cdot\vth}{r} = 2w_n(\vth),
\end{equation}
with $w_n$ satisfying~\eqref{eq:cgf_slope_n}.

Since the micro time step $\de t$ is stable, we can now take the limit as $n\goesto\infty$ on the right-hand side of~\eqref{eq:cgf_slope_n}. In consequence, we know that the point-wise limit $\om_\infty\doteq\lim_{n\to\infty}\om_n$ exists and, by \eqref{eq:invEM_var},  results in the following expression for the limiting function
\begin{equation}\label{eq:cgf_slope_infty}
w_{\infty}(\vth)=\frac{1}{2}\tp{\vth}\Big[\de t\sum_{j=0}^{\infty}(I+\de t A)^j B\tp{B}(I+\de t\tp{A})^j\Big]\vth = \frac{1}{2}\tp{\vth}V_\infty^{\de t}\vth.
\end{equation}

Having established the limiting behaviour of higher cumulants $\sder^j\cgf_n(0)$, for $j\geq3$, we will now use identity~\eqref{eq:cgf_slope_infty} to demonstrate that the second cumulants $\hess\cgf_n(0)$ converge to $V_\infty^{\de t}$. Via the fundamental theorem of calculus we have
\begin{equation*}
    \int_0^1r\tp{\vth}\hess\cgf_n(tr\vth)\vth\,\der{t} = \grad\cgf_n(r\vth)\cdot\vth - \grad\cgf_n(0)\cdot\vth.
\end{equation*}
Dividing by $r$ and changing variable of integration $t\mapsto t/r$ gives
\begin{equation*}
    \lim_{r\to+\infty}\frac{1}{r}\int_0^r\tp{\vth}\hess\cgf_n(t\vth)\vth\,\der{t} = \lim_{r\to+\infty}\Big(\frac{\grad\cgf_n(r\vth)\cdot\vth}{r} - \frac{\grad\cgf_n(0)\cdot\vth}{r}\Big) = 2w_n(\vth),
\end{equation*}
by the asymptotic property~\eqref{eq:grad_cgf_asym} of $r\mapsto\grad\cgf_n(r\vth)\cdot\vth$. Denoting $\mint{-}_0^{\infty}\doteq \lim_{r\to+\infty}\frac{1}{r}\int_0^r$, we have from~\eqref{eq:cgf_slope_infty}
\begin{equation*}
    \mint{-}_0^{\infty}\tp{\vth}\hess\cgf_n(t\vth)\vth\,\der{t} = 2w_{\infty}(\vth) = \tp{\vth}V^{\de t}_\infty\vth.
\end{equation*}
Because the Taylor expansion of $t\mapsto\cgf_n(t\vth)$ around $t=0$ gives $\hess\cgf_n(t\vth) = \hess\cgf_n(0) + \sder^3\cgf_n(\vth'_t)$, with some $\vth'_t$, and $\sder^3\cgf_n(\vth'_t)$ converges to zero as $n$ goes to $+\infty$, uniformly in $\vth'_t$, we get
\begin{equation*}
    \lim_{n\to+\infty}\tp{\vth}\hess\cgf_n(0)\vth = \lim_{n\to+\infty}\mint{-}_0^{\infty}\tp{\vth}\hess\cgf_n(t\vth)\vth\,\der{t} = \tp{\vth}V^{\de t}_\infty\vth,
\end{equation*}
for any $\vth\in\R^d$. This proves the convergence of the second cumulants.


In conclusion, since $\grad\cgf_n(0)=\mu_n$ converges to $0$, by assumption, $\hess\cgf_n(0)$ converges to $V_\infty^{\de t}$, and all higher order derivatives to $0$, we obtain the following limiting sequence of cumulants as $n$ goes to $+\infty$:
\begin{equation}
(0, V_\infty^{\de t}, 0,\dotsc).
\end{equation}
This sequence uniquely determines the normal distribution with mean zero and covariance matrix $V_\infty^{\de t}$. Thus, from the Frech\'{e}t-Shohat theorem~\cite[p.~307]{athreya2006measure}, the laws of $X_n$ converge weakly to $\nd_{0,V_\infty^{\de t}}$.
\end{proof}

\section{Numerical illustration: a periodically driven linear slow-fast system\label{sec:numerics}}

In this section, we numerically illustrate the convergence and stability results of the previous sections on an academic example. As in \cite{li2008effectiveness,vandecasteele2018efficiency},
 we define a linear slow-fast SDE with additive noise, in which we add a periodic forcing to the slow component 
\begin{equation} \label{eq:lineardrivensde}
\begin{cases}
dX = -2(X+Y)dt + \sin(2\pi t)dt + dW_x \\
dY = \frac{1}{\varepsilon}(X-Y)dt + \frac{1}{\sqrt{\varepsilon}}dW_y.
\end{cases}
\end{equation}
The parameter $\varepsilon$ is the time-scale separation between $X$ and $Y$ and controls the stiffness in the system. Introducing a periodic forcing allows to easily measure errors between the exact and numerical solution by computing the $L^2$ difference between both curves over one period. At the same time, the driving force does not alter stability and convergence properties of micro-macro acceleration. 

We first look at some numerical convergence tests (Section~\ref{sec:conv_num}), before moving on to computational experiments on stability (Section~\ref{sec:stab_num}).

\subsection{Convergence to the microscopic time integrator\label{sec:conv_num}}

By taking expectations of \eqref{eq:lineardrivensde}, one can show that the means $\mu_X$ (resp.~$\mu_Y$) of the slow (resp.~fast) component of the exact solution are given by
\begin{equation} \label{eq:lineardrivensolution}
\begin{pmatrix} \mu_X(t) \\ \mu_Y(t) \end{pmatrix} = e^{t M} \begin{pmatrix} \mu_{X_0}-A \\ \mu_{Y_0}-C \end{pmatrix} + \begin{pmatrix}A \\ C \end{pmatrix} \cos(2 \pi t) + \begin{pmatrix} B \\ D \end{pmatrix} \sin( 2\pi t),
\end{equation}
where $M = \begin{pmatrix} -2 & -2 \\ \ 10 & -10 \end{pmatrix}$ and $\begin{pmatrix}\mu_{X_0} & \mu_{Y_0} \end{pmatrix}^T$ is the mean of the initial condition of~\eqref{eq:lineardrivensde}. The constants $A,B,C$ and $D$ are the solution of the linear system 
\begin{equation*}
\begin{pmatrix} -2 \pi & 2 & 0 & 2 \\ 2 & 2\pi & 2 & 0 \\ 0 & -10 & -2 \pi & 10 \\ -10 & 0 & 10 & 2\pi \end{pmatrix} \begin{pmatrix} A \\ B \\ C \\ D \end{pmatrix} = \begin{pmatrix} 1 \\ 0 \\0 \\0 \end{pmatrix}.
\end{equation*}
To illustrate convergence (Theorem \ref{thm:linearslowconvergence}), we compute the error in the slow means obtained micro-macro acceleration against both the exact solution \eqref{eq:lineardrivensolution} and the numerical result obtained by the Euler-Maruyama integrator. We are mostly interested in the error between micro-macro acceleration and the Euler-Maruyama method because we want to understand the effect of extrapolation. As parameters we take $\varepsilon = 0.5$ and $\varepsilon = 0.05$ for the time-scale separation and we perform computations to an end time $T = 6$. We choose a small time step $\delta t = \varepsilon / 20$ for the microscopic time integrator and many values for the extrapolation time step $\Delta t$. The error is computed as the $L_2$ norm of the difference between two curves and averaged over 10 independent runs. The convergence results are depicted on Figure \ref{fig:convergence}.


\begin{figure}
    \centering
    \begin{subfigure}[b]{0.5\textwidth}
    \centering
    \includegraphics[width=\textwidth]{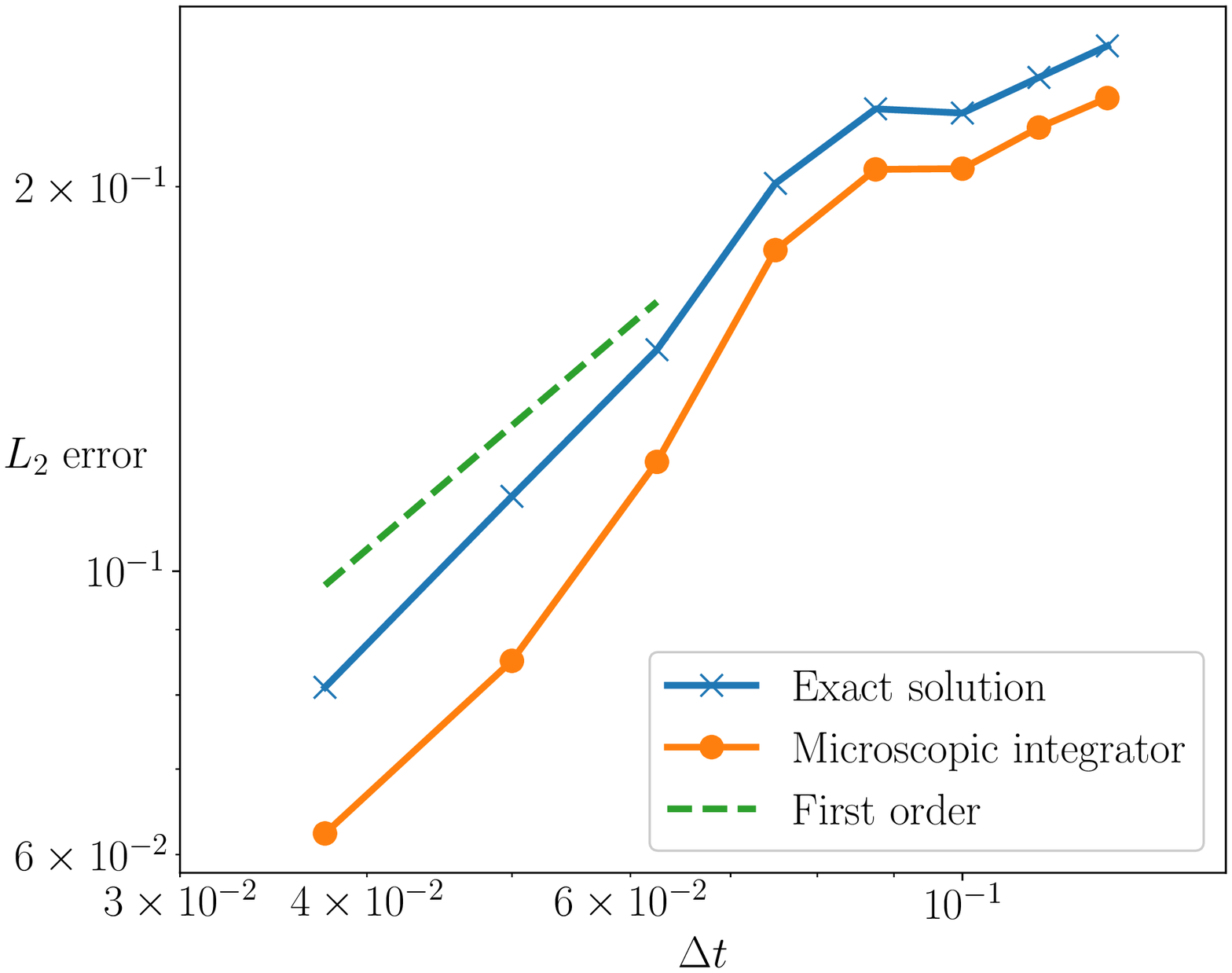}
    \end{subfigure}%
    \begin{subfigure}[b]{0.5\textwidth}
    \centering
    \includegraphics[width=\textwidth]{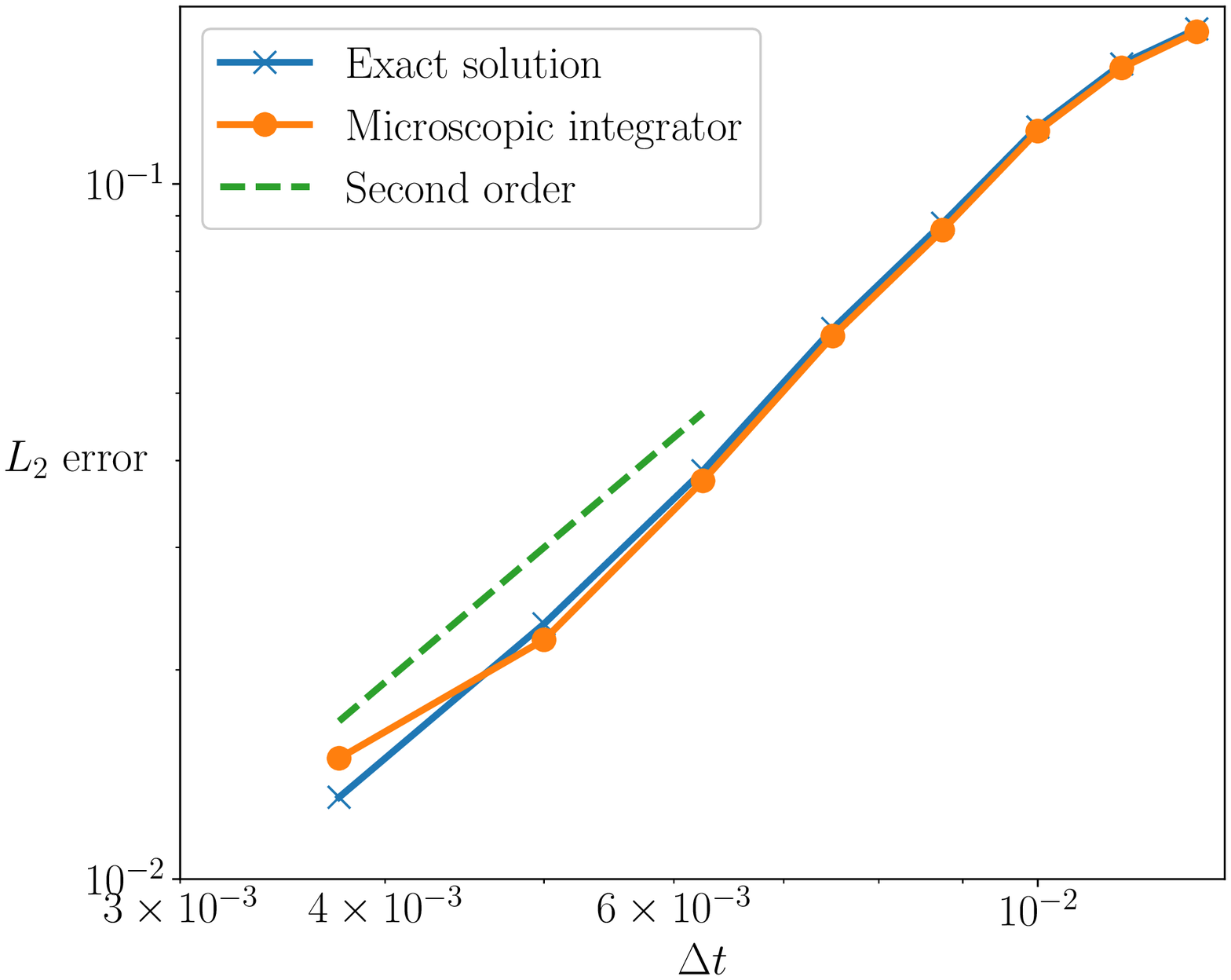}
    \end{subfigure}
    \caption{The error in the slow mean of micro-macro acceleration as a function of the extrapolation step size $\Delta t$, computed against the analytic solution \eqref{eq:lineardrivensolution} (blue) and against the numerical result obtained by the Euler-Maruyama method with time step $\delta t$ (orange) for $\varepsilon = 0.5$ (left) and $\varepsilon=0.05$ (right). We clearly see that the micro-macro acceleration error decreases when $\Delta t$ decreases, as given by Theorem \ref{thm:linearslowconvergence}. Moreover, for small $\varepsilon$ there is almost no difference between the error computed against the analytic solution and the microscopic time integrator as the latter is very accurate. This difference is higher for larger $\varepsilon$.}
    \label{fig:convergence}
\end{figure}
Figure \ref{fig:convergence} shows that the micro-macro acceleration error lowers as $\Delta t$ decreases to $\delta t$, as proven in Theorem \ref{thm:linearslowconvergence}. For a large $\varepsilon=0.5$ the error decreases linearly, while for a small $\varepsilon=0.05$ the decrease is quadratically. The order of convergence of micro-macro acceleration is not well-understood yet. In Figure \ref{fig:convergence} we also see that the error computed against the analytic solution and the Euler-Maruyama method is almost the same for small $\varepsilon=0.05$ since the microscopic integrator is very accurate. For larger $\varepsilon$, the error between micro-macro acceleration and the Euler-Maruyama method is smaller than with the exact solution because the Euler-Maruyama method makes a non-negligible error. 

\subsection{Stability with initial conditions having Gaussian tails\label{sec:stab_num}}

In the next experiment, we also consider the periodically driven linear system, but look at large extrapolation steps to study stability properties. The derivation in Section 4 does not give a threshold on the extrapolation step above which micro-macro acceleration becomes unstable and below which the algorithm is stable. To determine that a simulation was unstable, we will therefore rely on an alternative strategy, also proposed in \cite{debrabant2018study} for Gaussian initial conditions. In \cite{debrabant2018study}, it was shown that instability of the micro-macro acceleration technique unavoidably leads to so-called \textit{matching failures}, even before the solution blows up to infinity. A matching failure occurs when there exists no probability distribution that is consistent with the given macroscopic state variables $m$. In other words, the pair $(m, \mu)$ does not lie in the domain of the matching operator $\mathcal{M}$ for any prior distribution $\mu$, unless $\mu$ is consistent with $m$. 

In practice, we employ a Newton-Raphson method to compute the Lagrange multipliers in \eqref{eq:lagrangemultipliersmean} and detect a matching failure when the iterative method does not converge. Specifically, we must solve
\[
\int_G x \exp \left( \bar{\la}^s \cdot x - A(\bar{\la}^s, \mu) \right) d P^s(x) = \bar{\mu}^s,
\]
with a Newton-Raphson procedure, where we compute the integral using a Monte-Carlo representation of the prior distribution $P$. When the Newton-Raphson solver fails to reach the extrapolated slow mean $\bar{\mu}^s$ within an absolute tolerance of $10^{-11}$ in 50 steps, we mark a matching failure.

For the numerical experiment, we simulate the periodically driven linear system \eqref{eq:lineardrivensde} with $\epsilon=1$ for different pairs of step sizes $(\delta t, \Delta t)$ up to the end time $T = 100$. When a matching failure occurs, we mark the pair of parameters as unstable, and stable otherwise. The initial condition is Gaussian with mean zero and unit variance, which fits in the framework of Section 4. The number of Monte Carlo replicas is $N=10^5$ and we use $K=1$ step in the microscopic time integrator. The numerical results are summarized in Figure \ref{fig:stability}.

\begin{figure}
    \centering
    \includegraphics[width=0.4\textwidth]{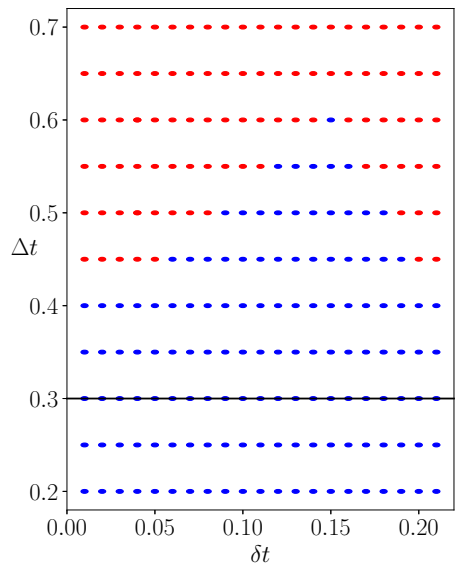}
    \caption{The $(\delta t, \Delta t)$ stability plane of micro-macro acceleration on the linear driven system \eqref{eq:lineardrivensde}. A blue dot indicates stability and a red dot instability. When at least one matching failure occurs, we mark an instability at the corresponding time step sizes. The stability domain has a V-shaped domain, and for every microscopic time step, the maximal extrapolation time step is larger than the deterministic stability bound of 0.3.}
    \label{fig:stability}
\end{figure}

Since $M$ is neither diagonal nor lower-triangular, there exist no stability bounds on the extrapolation step $\Delta t$ yet \cite{debrabant2018study}. As a proxy, we first compute the deterministic stability bound, as if there were no Brownian motion in \eqref{eq:lineardrivensde}. The eigenvalues of $M$ are $-6 \pm 2i$, implying the maximal deterministic time step is $\Delta t_{\text{max}} = 0.3$. The stability domain of micro-macro acceleration is V-shaped and for every $\delta t$, the maximal extrapolation step before instability is always greater than $0.3$. Micro-macro acceleration thus has good stability properties.

\section{Conclusion\label{sec:conclusions}}
We presented a micro-macro acceleration scheme, based on a combination of microscopic simulation and extrapolation of some macroscopic quantities of interest. We demonstrated that using only the slow mean during extrapolation results in a convergent and stable algorithm. The proofs hold for linear stochastic differential equations with additive noise. 
We complemented the analysis with numerical results  that indicate that the error of micro-macro acceleration decreases to zero when both the extrapolation and microscopic step size decrease to zero. For stability, we investigated for which pairs microscopic and extrapolation time steps micro-macro acceleration is stable, and compared the numerical results to the deterministic stability bounds in the case without Brownian motion. Empirically, the stability domain is V-shaped, and for every value of the microscopic time step, the maximal extrapolation step is above the deterministic stability bound.

\section*{Acknowledgements}
The authors thank Kristian Debrabant for proofreading the material in Section~\ref{sec:stability}. The authors acknowledge the support of the Research Council of the University of Leuven through grant ‘PDEOPT’ and of the Research Foundation – Flanders (FWO – Vlaanderen) under grant G.A003.13.

\bibliographystyle{plain}
\bibliography{ref}

\appendix

\section{Properties of cumulant generating functions}\label{se:cgf}
First, we mention the convexity of $\cgf_{P}$.
\begin{proposition}[{\cite[Thm.~A.4]{jorgensen1995exponential}}]\label{pro:cgf_conv}
Let $P\in\prob^d$. Then
\begin{enumerate}[{\normalfont (i)}]
\item\label{pro:cgf_conv_dom}
The set $\Th_P$ is convex.
\item\label{pro:cgf_conv_cgf}
$\cgf_P$ is a~convex function on $\Th_P$, and strictly convex if and only if $P$ is not concentrated in a~single point.
\end{enumerate}
\end{proposition}

When $X\sim P$ we also write $\cgf_X$ instead of $\cgf_P$, and $\Th_X$ instead of $\Th_P$.
\begin{proposition}[{\cite[Thm.~A.1 \&~A.7]{jorgensen1995exponential}, \cite[31]{keener2011theoretical}}]\label{pro:cgf_prop}
Assume that $X\sim P$, where $P\in\prob^d$. Then
\begin{enumerate}[{\normalfont (i)}]
\item\label{pro:cgf_prop_bound}
$-\infty<\cgf_P(\vth)\leq+\infty$ for $\vth\in\R^d$.
\item\label{pro:cgf_prop_zero}
$\cgf_P(0)=0$.
\item\label{pro:cgf_prop_lintran}
If $M\in\R^{l\times d}$ and $c\in\R^l$
\begin{gather*}
\begin{aligned}
\Th_{MX+c} &= \{s\in\R^l: \tp{M}s\in\Th_X\}\\[.2em]
\cgf_{MX+c}(s) &= \cgf_X(\tp{M}s) + s\sdot c,\quad s\in\R^l.
\end{aligned}
\end{gather*}
\item\label{pro:cgf_prop_tayl}
If $0\in\inter\Th_P$, $\cgf_P$ is analytic on $\inter\Th_P$ with Taylor expansion around $0$
\begin{equation*}
\cgf_P(\vth) = \sum_{j=1}^{+\infty}\frac{1}{j!}\sder^j\cgf_P(0)[\vth],
\end{equation*}
where $\sder^j\cgf_P(0)$ is the $j$th cumulant of $P$, considered as the $j$-linear mapping on $\R^d$, and $\sder^j\cgf_P(0)[\vth]\doteq\sder^j\cgf_P(0)[\vth,\dotsc,\vth]$. In particular, $X$ has vector mean and covariance matrix
\begin{equation*}
\Exp[][X]=\grad\cgf_{P}(0),\quad \Var(X)=\hess\cgf_{P}(0).
\end{equation*}
\item\label{pro:cgf_prop_indep}
If $X,Y$ are independent, $\cgf_{X+Y}=\cgf_X+\cgf_Y$ and $\Th_{X+Y}=\Th_X\cap\Th_Y$.
\end{enumerate}
\end{proposition}

\begin{proposition}\label{pro:cgf_prop_match}
If $Q=\match(\oline{\mu}^s,P)$ is the solution to \eqref{eq:slowmeanmatching} with a~(slow) marginal mean $\oline{\mu}^s\in\R^{d_s}$ and a~prior $P\in\prob^d$, then
\begin{equation*}
\cgf_Q(\vth)=\cgf_P(\vth+\oline{\la}^s\!\oplus0^f) - \cgf_P(\oline{\la}^s\!\oplus0^f),
\end{equation*}
where $\oline{\la}^s\in\R^{d_s}$ is a~vector of Lagrange multipliers corresponding to $\oline{\mu}^s$ and $0^f=(0,\dotsc,0)\in\R^{d_f}$.
\end{proposition}
\begin{proof}
Using formula~\eqref{eq:slowmeanmatching} we compute
\begin{align*}
\cgf_{Q}(\vth)
&= \ln\Exp[P]\Big[e^{\vth\sdot\proj}\,\rnder{Q}{P}\Big]\\
&= \ln\Exp[P]\Big[e^{\vth\sdot\proj+\oline{\la}^s\!\sdot\proj^s}\, e^{- A(\oline{\la}^s,P^s)}\Big]\\
&= \ln\Exp[P]\Big[e^{(\vth+\oline{\la}^s\!\oplus0^f)\sdot\proj}\Big] - A(\oline{\la}^s,P^s).
\end{align*}
It remains to note that by the definitions of log-partition function and the marginal distribution
\begin{equation*}
A(\oline{\la}^s,P^s) = \ln\Exp[P^s]\Big[e^{\oline{\la}^s\sdot\proj^s}\Big] = \ln\Exp[P]\Big[e^{(\oline{\la}^s\!\oplus0^f)\sdot\proj}\Big].\qedhere
\end{equation*}
\end{proof}

\begin{corollary}\label{co:match_mean_ndprior}
For any vectors $\oline{\mu},\mu\in\R^d$ and a symmetric, non-negative definite matrix $\Si\in\mat{d}{d}$, $\match(\oline{\mu},\nd_{\mu,\Si})=\nd_{\oline{\mu},\Si}$.
\end{corollary}
\begin{proof}
According to \eqref{eq:lagrangemultipliersmean}, the vector of Lagrange multipliers $\oline{\la}$ corresponding to matching with $\oline{\mu}$ satisfies $\grad[\la]A(\oline{\la},\nd_{\mu,\Si})=\oline{\mu}$. When matching with mean only, the log-partition function coincides with the cumulant generating function, that is $A(\la,\nd_{\mu,\Si})=\cgf_{\nd_{\mu,\Si}}(\la)$, and using formula from Example~\ref{ex:cgf_nd} we obtain
\begin{equation*}
\Si\oline{\la}=\oline{\mu}-\mu.
\end{equation*}
Employing Proposition~\ref{pro:cgf_prop_match} with $Q=\match(\oline{\mu},\nd_{\mu,\Si})$ we compute
\begin{align*}
\cgf_{Q}(\vth)
&= \mu\sdot\vth + \frac{1}{2}\big(\tp{\vth}\Si\oline{\la} + \tp{\oline{\la}}\Si\vth\big) + \frac{1}{2}\tp{\vth}\Si\vth\\
&= \mu\sdot\vth + (\oline{\mu}-\mu)\sdot\vth + \frac{1}{2}\tp{\vth}\Si\vth\\
&= \oline{\mu}\sdot\vth + \frac{1}{2}\tp{\vth}\Si\vth.\qedhere
\end{align*}
\end{proof}

\end{document}